\documentclass[11pt]{amsart}

\usepackage{graphicx}
\usepackage{amssymb}
\theoremstyle{definition}

\newtheorem{theorem}{Theorem}[section]

\newtheorem{definition}[theorem]{Definition}
\newtheorem{example}[theorem]{Example}

\newtheorem{proposition}[theorem]{Proposition}

\numberwithin{equation}{section}

\DeclareMathAlphabet\scr{U}{scr}{m}{n}
\SetMathAlphabet\scr{bold}{U}{scr}{b}{n}
  \DeclareFontFamily{U}{scr}{\skewchar\font'177}%
  \DeclareFontShape{U}{scr}{m}{n}{<-6>rsfs5<6-8>rsfs7<8->rsfs10}{}%
  \DeclareFontShape{U}{scr}{b}{n}{<-6>rsfs5<6-8>rsfs7<8->rsfs10}{}%

\renewcommand{\Re}{\mathrm{Re}}

\renewcommand{\epsilon}{\varepsilon}
\renewcommand{\theta}{\vartheta}
\renewcommand{\rho}{\varrho}

\begin{document}
\title[]{A note on the Esscher transform of affine Markov processes}

%\author[E.~Mayerhofer]{Eberhard Mayerhofer}
%\address{Dublin City University} \email{eberhard.mayerhofer@dcu.ie}
%
%\thanks{Financial support from Grant ERC 278295,
%and SFI  08/SRC/FMC1389 is gratefully acknowledged.}
\maketitle
\begin{abstract}
In affine models, both the martingale property of stochastic exponentials and non-explosion of affine processes can be characterized
in terms of minimality of solutions to a system of generalized Riccati differential equations. We improve these characterizations for affine processes on 
$\mathbb R_+^m\times\mathbb R^n$ by Mayerhofer, Muhle-Karbe and Smirnov (2011) and Keller-Ressel and Mayerhofer (2014)  by showing that 
the characterizing minimal solution is the unique one.
%
%: A process is conservative if and only if
%a related Riccati differential equation with trivial initial data does not admit non-trivial solutions. By an equivalent change of measure we turn this result into a similar improvement of the martingale characterization of exponentially affine functionals by Keller-Ressel and Mayerhofer (2014).
%Mathematics Subject Classification (2000):
\end{abstract}

\section{Introduction}
Affine Markov Processes constitute a fairly general class of Markov processes. They comprise, among others,
the class of L\'evy-processes, OU-type processes (with linear drift) driven by L\'evy noise, the well known
Feller diffusion \cite{feller51}, many celebrated interest rate models (\cite{daisingleton00}, \cite{duffiekan96},  \cite{duffie2000transform}), stochastic volatility models in finance  (Such as Bates' \cite{Bates2000}, Heston's \cite{Heston1993} and Barndorff-Nielsen,  \& Shepard's \cite{Barndorff2001}), and models of default risk, e.g. \cite{garleanu}. Their extensive use in finance has motivated stochastic processes research in the last two decades. Existence Theories for general affine processes on particular state spaces have been established: We mention the one for canonical state spaces $\mathbb R_+^m\times\mathbb R^n$ \cite{dfs} and for positive semi-definite matrices $S_d^+$ \cite{cfmt}.

Undoubtably, the class of L\'evy processes \cite{Sato1999} is the largest subclass of affine processes with a well known and well established theory. 
Affine processes allow differential semimartingale characteristics which are affine in the state variable. This more general dynamic behaviour complicates their analysis, and distinguishes them from the L\'evy-class e.g. in the following regards: 
\begin{enumerate}
\item \label{issue 1xx} A L\'evy process is non-explosive, if and only if has zero potential. Affine processes, however, may explode
e.g., due to increasing (along every path) jump intensities, see e.g., \cite[Example 9.3 ]{dfs}.
\item \label{issue 2xx} A L\'evy process has finite exponential moment, if and only the associated L\'evy-measure does \cite[Theorem 25.17]{Sato1999}. This entails, for instance, that the moment explosion is a time independent feature. In affine models, however,
moment explosion may occur in finite time. For instance, the Feller diffusion \cite{feller51} has non-centrally chi-square distributed
marginal distributions, hence does not admit finite exponential moments of all orders. More interesting examples are provided by
jump-diffusions.
\item \label{issue 3xx} If $L=(L_t)_t$ is a conservative L\'evy process with $L_0=0$ a.s., and $g(t,\theta):=\mathbb E[e^{\theta^\top L_t}]<\infty$ for $\theta\in\mathbb R^d$ then the process
\begin{equation}\label{eq: levy martingale}
M_t:=e^{\theta ^\top L_t}/g(t,\theta)
\end{equation}
is a true martingale (this is the so-called Esscher transform). This is not true for affine processes, in general, due to time-dependent moment explosion.
%\item Not all affine processes have infinitely divisible marginal distributions. For instance, on $S_d^+$ only pure jump processes are infinitely divisible \cite[Theorem 2.9]{cfmt}. A well known example for a diffusion process which is not infinitely divisible is
%the Wishart processes \cite{bru}: Its transition laws are not infinitely divisible \cite{LetacMassam}.
\end{enumerate}

This paper deals expands on \ref{issue 1xx} and \ref{issue 3xx}for affine processes. In this setting, the martingale property of stochastic exponentials and non-explosion of affine processes can be characterized
in terms of minimality of solutions to a system of generalized Riccati differential equations. We improve the characterization of conservativeness for affine processes on 
$\mathbb R_+^m\times\mathbb R^n$ by \cite{dfs} and \cite{MKEM} by showing that minimal solutions are actually unique. In Theorem \ref{main theorem}, we show that a process is conservative if and only if
a related Riccati differential equation with trivial initial data does not admit non-trivial solutions. Using an exponential tilting technique we turn this result into a similar improvement of the martingale characterization of exponentially affine functionals of affine processes by \cite{MKREM}: For conservative affine processes $(X_t)_t$ starting at $x\in\mathbb R^d$, the natural generalization of  L\'evy martingales of the
form \eqref{eq: levy martingale} is given by processes of the form 
\[
M_t=e^{\theta ^\top (X_t-x)-\int_0^t \lambda^\top X_s ds}
\]
with an appropriately chosen $\lambda\in\mathbb R^d$. Not always such processes are either well defined for all times, nor do they always give 
rise to martingales. The characterization of the martingale property of $M_t$ is given by Theorem \ref{main second theorem}.
Nevertheless, an Example in the final section \ref{final sec} demonstrates that for a characterization of the validity of the affine transform formula and thus of moment explosion \ref{issue 2xx}
in terms of solvability of the associated generalized Riccati equations, the concept of minimal unique solutions as introduced in \cite{MKREM} is not redundant.

\section{Setting}

In this paper, $(X,P_x)_{x\in D}$ is a stochastically continuous, affine Markov process with state space
$D=\mathbb R_+^m\times\mathbb R^n$, see, e.g. \cite{dfs}. This means, there are functions $\phi(t,u),\psi(t,u)$ such that
\begin{equation}\label{eq aft}
\mathbb E[e^{\langle u, X_t\rangle }\mid X_0=x]=e^{\phi(t,u)+\langle \psi(t,u), x\rangle}
\end{equation}
for all $u\in i\mathbb R^d$, $t\geq 0$ and $x\in D$. Here $\langle \cdot, \cdot \rangle$ denotes the standard
Euclidean scalar product on $\mathbb R^d$. Stochastic continuity of $X$ implies that
$(\phi,\psi)$ are differentiable  \cite{regu1}, \cite{regu2}  with functional characteristics 
\[
F(u)=\left.\frac{\partial \phi(t,u)}{\partial t}\right\vert_{t=0},\quad R(u)=\left.\frac{\partial \psi(t,u)}{\partial t}\right\vert_{t=0}.
\]
Furthermore, $\phi,\psi$ solve the generalized Riccati differential equations
\begin{align}\label{eq ric1}
\partial_t\psi(t,u)&=R(\psi(t,u),\quad \psi(0,u)=u,\\\label{eq ric2}
\partial_t\phi(t,u)&=F(\psi(t,u),\quad \phi(0,u)=0,
\end{align}
where $F$ and $R=(R_1,R_2,\dots, R_d)$ are of L\'evy-Khintchine form on $\mathbb R^d$. In the following, we abbreviate the index sets $I=\{1,\dots,m\}$ and $J=\{1,\dots,m\}$, and for a vector $x\in \mathbb R^d$
we write $x=(x_I,x_J)$, where $x_I=(x_1,\dots,x_m)$
In particular\footnote{
 The precise parametric conditions can be found in \cite[Definition 2.6]{dfs}. For the sake of simplicyt we use the minimal information necessary to derive our conclusions.}, for a positive semidefinite matrix $a$, $b\in D$, and a L\'evy measure $\mu_0$ on $D$, we have
\[
F(u)=\langle a u, u\rangle+\langle b,u\rangle+\int_{D\setminus 0}(e^{\langle u,\xi\rangle}-1-\langle \chi(\xi),u\rangle)\mu_0(d\xi),
\]
where $\chi$ is a truncation function.

Also, for $i=1,\dots,m$, and for $u\in i\mathbb R^d$, the component $R_i$ can be written as
\begin{equation}\label{Ri}
R_i(u)=\alpha_i u_i^2+\langle \beta_i,u\rangle+\int_{D\setminus 0}(e^{\langle u,\xi\rangle}-1-\langle \chi_i(\xi),u\rangle)\mu_i(d\xi),
\end{equation}
where $\alpha_i\in\mathbb R_+$, $\beta_i\in\mathbb R^d$ with $\beta_{i,j}\geq 0$ for $j\neq i$, $1\leq j\leq m$. Here 
we use the truncation function $\chi_i$ from \cite{dfs}, which is a function $D\setminus\{0\}\rightarrow \mathbb R^d$
defined in components by 
\[
\chi_{i,j}(\xi)=(\xi_j \wedge 1)\frac{\xi_j}{\vert \xi_j\vert} \quad\text{for}\quad i\in J\cup\{i\},
\]
and $\chi_{i,j}\equiv 0$ for $j\in I\setminus\{i\}$.

Since $\mu_i\equiv 0$ for $i\in J$, we may adapt the definition of \cite[ Equation(5.1)]{MKREM}
\[
\mathcal Y:=\{y\in\mathbb R^d\mid \sum_{i=0}^m\int_{\|\xi\|\geq 1}e^{\langle y,\xi\rangle}\mu_i(d\xi)<\infty\}.
\]
which is the intersection of the effective domains of $F,R_1,\dots,R_d$. For a generic affine process, we only know a-priori that
 $0\in \mathcal Y$. But if $\mathcal Y$ has non-empty interior, $\mathcal Y^\circ\neq\varnothing$, then
on $\mathcal Y^\circ$, the functions $F, R$ are analytic. The last assumption will enter in section \ref{sec martingales} only.

The partial order induced by $R_+^m\times\mathbb R^n$ is denoted by $\preceq$ and is given by
$u\preceq v$ if and only if $u_i\leq v_i$ when $i\in I$ and $u_J=v_J$

We define a special property of a domain in $\mathbb R^d$:
\begin{definition}
A set $U\subset\mathbb R^d$ is order preserving, if for $u,v\in\mathbb R^d$ such that $v\in U$, we have
$u\preceq v$ imples $u\in U$. 
\end{definition}
From the definition of $R$ in \eqref{Ri} we can easily conclude that the domain $\mathcal Y$ is order preserving (for a similar statement for $D=\mathbb R_+^m$, see \cite{AGS}).

\section{Conservative Affine Processes}
 It can be shown that 
\[
\psi_J(t,u)=e^{\beta^\top_{JJ}}u
\]
for a real $n\times n$ matrix $\beta_{JJ}$, and that when $u_J=0$, the first $m$ components $\psi_I$ of $\psi$ satisfy the generalized Riccati differential equation
\begin{equation}\label{eq: reduced}
\partial_t g(t)=\widetilde R(g(t)),
\end{equation}
where $\widetilde R=R_I(u_I,0)$.

This section provides the following improvement of \cite[Theorem 3.4]{MKEM}:
\begin{theorem}\label{main theorem}
The following are equivalent:
\begin{enumerate}
\item \label{point 1} $(X,P_x)_{x\in D}$ is conservative.
\item \label{point 2} $F(0)=0$ and the only solution of $g(t)$ of \eqref{eq: reduced} with $g(0)=0$ is the trivial one.
\end{enumerate}
\end{theorem}
The crucial point here is that we do not restrict ourselves to uniqueness only among $\mathbb R_-^m$--valued solutions
of \eqref{eq: reduced}, as in \cite{MKEM} or in \cite{dfs} (where the restriction is made to solutions being strictly negative
for $t>0$).

\begin{figure}[h]
\centering
\includegraphics[width=1.1\textwidth]{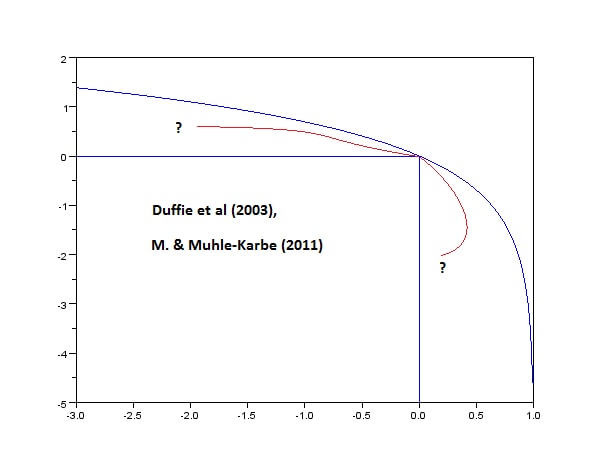}
\caption{\label{fig:front} A typical convex, order-preserving domain $U$  in $\mathbb R^2$, where equation \eqref{eq: reduced} may be defined. In ther conservative case, \cite{dfs} and \cite{MKREM} rule out any non-trivial solutions inside the rectangle. We show that no solutions such as the suggested
red curves exists (Proposition \ref{prop superbus}).
In general, the boundary of the domain $U$ is more complicated (Theorem \ref{main theorem})}
\label{fig:mvt}
\end{figure}

 A step to this result is the following special multivariate ODE comparison result, which is more general than \cite[Proposition 3.3]{MKEM}:
%\footnote{The reason is that \cite{MKEM}
%restricts to solutions with values in $\mathbb R_-^m$}.%, and a characterization of finite exponential moments \cite{MKREM}. 
\begin{proposition}\label{prop superbus}
Let $T>0$ and $u_I\in\mathbb R_-^m$. If  $g(t): [0,T)$ is a solution of \eqref{eq: reduced}, then $g(t)\succeq \psi_I(t,(u_I,0))$, for
all $t<T$.
\end{proposition}
\begin{proof}
%A deterministic proof, using sophisticated comparison arguments for vector fields which might fail to be locally Liptschitz at the boundary of a domain,
%is given in \cite{MKREM}. For a stochastic proof, 
We introduce
\[
\widetilde M(t,X_t):=e^{f(T-t)+\langle X_{I,t},g(T-t)\rangle}
\]
where $f(t)=\int_0^{t} F((g(s),0))ds$. Clearly $(f,(g,0))$ solve the generalized Riccati differential equations
\eqref{eq ric1}--\eqref{eq ric2}. By \cite[Theorem 4.7.1]{ethier2009markov} and a localizing argument, the process
\[
\widetilde M(t,X_t)-\int_0^t \partial_s \widetilde M(s,X_s)+\mathcal A \widetilde M(s,X_s) ds
\]
is a local martingale, where $\mathcal A$ is the infinitesimal generator of $X$; furthermore by construction we have
\[
 \partial_t\widetilde M(t,X_t)+\mathcal A \widetilde M(t,X_t)=0,
\]
for each $t$, a.s., because for any $u\in\mathbb R^d$ and $f_u(x):=\exp(\langle u,x\rangle)$ we have
\[
\mathcal Af_u(x)=f_u(x)(F(u)+\langle R(u),x\rangle).
\]
So $\widetilde M(t)$ is a positive local martingale, hence a supermartingale\footnote{For an alternative proof using the explicit semimartingale decomposition
of $X$, please consult \cite[Proof of Proposition 4.6]{MKREM}}. On the other hand,
the process
\[
M(t,X_t):=e^{\phi(T-t,(u_I,0))+\langle X_{I,t},\psi_I(T-t)\rangle}
\]
is a martingale satisfying $M(T)=\widetilde M(T)$. It follows that for each $t\in[0,T]$, $M(t)\leq \widetilde M(t)$, and therefore, by taking logarithms, we obtain
$g_i(t)\geq \psi_i(t)$, for $t\in[0,T)$ each $i=1,\dots,m$.
\end{proof}
%\begin{remark}
%The condition $u_I\in\mathbb R_-^m$ in Proposition \ref{prop superbus} can be replaced by $u_I$ such that
%\[
%\mathbb E[e^{\langle u, X_T\rangle}\mid X_0=x]<\infty
%\]
%where we set $u=(u_I,0)$. Then the assertion of Proposition \ref{prop superbus} follows from \cite[Theorem 2.14]{MKREM}.
%However, not without further assumptions. The framework of \cite{MKREM} insists that $(X,P_x)_{x\in D}$ is conservative and that the domain, where $\phi,\psi$
%are solved, has non-empty interior. The last condition is not imposed in the present paper.
%\end{remark}
\subsection{Proof of Theorem \ref{main theorem}}
In view of \cite[Proposition 9.1]{dfs} or \cite[Theorem 3.4]{MKEM}, only the direction
\ref{point 1} $\Rightarrow $ \ref{point 2} need to be proved. Suppose, for a contradiction, there exists a local solution
$g(t)$ on $[0,T)$ which is non-trivial. By Proposition \ref{prop superbus} we have $g\succeq 0$, hence there exist $1\leq r\leq m$ components of $g$ which are not identically zero on $(0,T)$, while the last $m-r$ actually vanish identically zero
on this interval. We introduce
the index set $K=(1,\dots,r)$. Without loss of generality we may assume that for $t\in (0,T)$, $g_i(t)>0$
for $i\in K$ and $g_j(t)\equiv 0$ for $j\in I\setminus K$. Let us choose $t_1,\dots,t_r$ be in $(0,T)$, where $g_i(t_i)\neq 0$, $i\in K$. Set $T':=\min \{t_i\mid i\in K\}>0$. Knowing that $g_{i}(t)\equiv 0$ for $t\in [0,T')$ and all $i\in I\setminus K$, 
we see that $h(t):=g_{I\setminus K}(t)$ satisfies an autonomous differential equation
\begin{equation}\label{eq: reduced12}
\partial_t h(t)=\tilde R(h(t)),\quad h(0)=0
\end{equation}
where $\tilde R=R_K(u_K,0)$. For each $i\in K$ $\tilde R_i(u_K)$ is given by
\begin{equation}\label{eq ric K}
\tilde R_i(u_K)=\alpha_ i u_i^2+\langle\beta_{i,K},u_K\rangle+\int_{D\setminus \{0\}}(e^{\langle u_K, \xi_K\rangle}-1-(\xi_i\wedge 1) u_i)\mu_i(d\xi).
\end{equation}
with $\alpha_i\geq 0$, $\beta_{i,i}\in\mathbb R$, $\beta_{i,j}\geq 0$ for $j\neq i$, and
$\mu_i$ are positive, sigma-finite measures on $D\setminus\{0\}$ which integrate
\[
h(\xi)=(\|\xi_{I\setminus\{i\}}\|\wedge 1)(\|\xi_{J\cup \{i\}}\|\wedge 1)^2.
\] 
Let $\theta=(\theta_i)_{i\in K}$, where $\theta_i:=g_i(t_i)$, $i\in K$. Since by \eqref{eq: reduced12}
\[
\vert\int_{D\setminus \{0\}}(e^{\langle \theta, \xi_K\rangle}-1-(\xi_i\wedge 1) \theta_i)\mu_i(d\xi)\vert<\infty,
\]
%\]
%\begin{align*}
%&\vert\langle\theta_K,\xi_K\rangle\vert^2\int_0^1 (1-s) e^{s \langle\theta_K, \xi_K\rangle}ds=\\
%&=e^{\langle \theta_K, \xi_K\rangle}-1-(\xi_i\theta_i)\leq e^{\langle\theta_K, \xi_K\rangle}-1-\langle\theta_K, \xi_K\rangle<\infty
%\end{align*}
%and
%\[
%e^{\theta_K^\top \xi_K}-1=\theta_K^\top \xi_K \int_0^1 e^{s \theta_K^\top \xi_K}ds
%\]
we conclude that for
\[
\mathcal Y_*:=\{ u\in \mathbb R^{m+n}\mid \int_{\|\xi\|>1} e^{\langle u, \xi\rangle}\mu_i(d\xi)<\infty \quad i\in I\} 
\]
we have
\begin{equation}\label{eq inclusion}
C:=\{u\in\mathbb R^d\mid u_i\leq \theta_i, i\in K, u_{J\cup (I\setminus K)}=0\}\subseteq \mathcal Y_*
\end{equation}

Let us define the following $r$ L\'evy measures on $\mathbb R^r$,
\[
\hat\mu_i(A)=\int_{A\times\mathbb R^{m-r}}\mu_i(d\xi),\quad i\in K,
\]
where $A$ is a generic Borel set in $\mathbb R^r$ whose closure is supported away from $0$. In view of \eqref{eq inclusion} these measures are well defined. Using the latter, we define a new system of generalized Riccati differential equations on $\mathbb R^r$,
\begin{equation}\label{eq: reduced1}
\partial_t h(t)=\hat R(h(t)),\quad h(0)=0
\end{equation}
where $\hat R$ is derived from $\widetilde R(u_k)=R_K(u_K,0)$ in \eqref{eq ric K} by condensing $\mu_i$ into $\hat \mu_i$, for $i=1,\dots,r$: For $v\in\mathbb R^r$, we set
\[
\hat R(v)=\alpha_ i u_i^2+\langle\beta_{i,K},u_K\rangle+\int_{K\setminus \{0\}}(e^{\langle u_K, \xi_K\rangle}-1-(\xi_i\wedge 1)u_i)\hat\mu_i(d\xi).
\]

%By construction, $\tilde \mu_i$ are admissible (a fact that allows us to define a homogeneous affine
%process $Y$ with characteristics $\hat F=0,\hat R$, but we shall not need this fact, shall we?). 

By construction, we know $h(t)=(g_1(t),\dots,g_r(t))$
is a non-trivial solution of \eqref{eq: reduced1}. We know that
$\hat R$ is analytic at least on the set 
\[
\{ u\in\mathbb R^m\mid u_i\leq \theta \quad\text{for all}\quad i\in K\},
\]
where $\theta:=\min\{g_i(t_i),i=1,\dots,r\}>0$. That means that $\hat R$ is analytic in an open neighborhood of the origin. By
the standard existence and uniqueness theorem of ODEs, we must have $h\equiv 0$ on $[0,T')$,  a mere impossibility.
We have therefore proven uniqueness for solutions
of the IVP \eqref{eq: reduced} with initial data $g(0)=0$, and we are done.

\section{True Exponentially Affine Martingales}\label{sec martingales}
Suppose $(X,P_x)_{x\in D}$ be a conservative affine process on $D=\mathbb R_+^m\times\mathbb R^n$. In this section we are interested in characterizing the martingale property of discounted exponentially affine functionals of the form
\[
\tilde S_t= e^{-\int_0^t L(X_s)ds+ \langle \theta, X_t\rangle},
\]
for $\theta\in\mathbb R^d$, and the linear functional $L(x)=l+\langle \lambda, x\rangle$, $l\in\mathbb R,\lambda\in\mathbb R^d$.

Let us define the extended functional characteristics 
$R'(u)=R(u)-\lambda$, and $F'(u)=F(u)-l$, and the corresponding generalized Riccati differential equations
\begin{align}\label{ric ex 1}
\partial_t\psi(t,u)&= R'(\psi(t,u)),\quad \tilde\psi(0,u)=u,\\\label{ric ex 2}
\partial_t\phi(t,u)&= F'(\psi(t,u)),\quad \tilde\phi(0,u)=0.
\end{align}
For affine processes on canonical state spaces, the following is an improvement of 
\cite[Theorem 3.1 (2)]{MKREM}, because it is stated without using the notion of minimal solutions:
\begin{theorem}\label{main second theorem}
Let $x\in D^\circ$. The process $(\tilde S_t)_{t\geq 0}$ is a true $P_x$-martingale if and only if
$\theta\in\mathcal Y$, $F(\theta)=l$, $R(\theta)=\lambda$ and $\phi(\cdot,\theta)\equiv 0$ and $\psi(\cdot,\theta)\equiv\theta$
 are the unique solutions of the Riccati equations \eqref{ric ex 1}--\eqref{ric ex 2}.
\end{theorem}
\begin{proof}
By \cite[Theorem 3.1 (2)]{MKREM}  the assertion holds when the term 'unique solutions' is replaced by
'minimal unique solution'.

Therefore, to prove the present assertion, it suffices to show that for  $\theta\in\mathcal Y$, $F(\theta)=l$, $R(\theta)=\lambda$, the unique minimal solutions $\phi(\cdot,\theta)\equiv 0$ and $\psi(\cdot,\theta)\equiv\theta$, of the Riccati equations \eqref{ric ex 1}--\eqref{ric ex 2}, are actually unique.

The differential equation \eqref{ric ex 2} for $\phi$ is a trivial once. Therefore it suffices to show that \eqref{ric ex 1} yields unique solutions. 
%
%An inspection of the proof of \cite[Theorem 4.14]{keller} shows that its conclusion holds true also with our slightly different assumptions. To this end it is useful to recall
%that we know that
By \cite[Theorem 3.1 (2)]{MKREM}, we also have that
\[
M_t^x=\exp\left(\langle \theta, X_t-x\rangle-t F(\theta)-\langle R(\theta),\int_0^t X_s ds\rangle\right)
\]
is a $P_x$-martingale on $[0,\infty)$ with expectation $1$ and
$M_t^x=e^{-\langle \theta,x\rangle}\tilde S_t$. 

Following the lines of the proof of \cite[Theorem 4.14]{keller} concerning exponential tilting let us then conclude there exist measures $Q_x\sim P_x$, for each $x\in D$, such that
$(X_t,\mathbb Q^x)_{t\geq 0,x \in D}$ is an affine process with characteristics
\[
\tilde F(u)=F(u+\theta)-F(\theta),\quad \tilde R(u)=R(u+\theta)-R(\theta)
\]
with real domain $\mathcal Y-\theta$. Furthermore for every $t\in\mathbb R_+^m$, $x\in D$, and $A\in\mathcal F_t$, it holds that
\[
\mathbb E^{Q_x}[1_{A}]=\mathbb E^{P_x}[1_{A}M_t^x],
\]
In particular, we have $(X,Q_x)_{x\in D}$ is conservative. By Theorem \ref{main theorem},
the system of generalized Riccati differential equations
\begin{align}\label{eq ric1 tilde}
\partial_t\tilde\psi(t,u)&=\tilde R(\tilde\psi(t,u),\quad \tilde\psi(0,u)=0,\\\label{eq ric2 tilde}
\partial_t\tilde\phi(t,u)&=\tilde F(\tilde\psi(t,u),\quad \tilde\phi(0,u)=0,
\end{align}
has a unique solution, namely the trivial one. Assume, for a contradiction, there exists a solution $\zeta(t)\neq \theta$
of  \eqref{ric ex 1} with $R(\theta)=\lambda$. Define $\tilde \psi(t):=\zeta(t)-\theta$. Then
$\tilde \psi(t)$ is a solution of \eqref{eq ric1 tilde}, and by assumption $\tilde \psi(t)$ does not vanish identically.
This contradicts Theorem \ref{main theorem}.

\end{proof}

\section{Non-uniqueness for Riccati differential equations}\label{final sec}
In the previous sections, we have shown that a characterizing minimal solution is actually the unique solution of a particular generalized Riccati differential equation. In general, however, we do not have uniqueness, but the affine transform formula \eqref{eq aft} only holds for the unique minimal solution of the associated generalized Riccati differential equations \eqref{eq ric2}--\eqref{eq ric2} (the general theory for real exponential moments is provided in \cite{MKREM}).

The following example of a homogeneous affine pure-jump process is used by \cite{KR2014} for the construction of exponentially affine strict local martingales
and is inspired by Example 9.3 of Duffie et al (2003). Here the linear jump characteristic $\mu(d\xi)$ is the L\'evy measure of a self-decomposable distribution.
Non-uniqueness of the initial value problem below is due to the lack of regularity of the vector field $R(u)$ at the boundary of the moment generating function.
\begin{example}
Let $(X,\mathbb P_x)_{x\in D}$, where $D=\mathbb R_+$, be an affine pure-jump process with
negative linear drift
\[
\beta(x)=-\sqrt{\pi}x
\]
and absolutely continuous compensator
\[
\nu( d\xi, ds):=X_{s-}\mu(d\xi)ds,\quad \mu(d\xi):=\frac{2}{\sqrt\pi}e^{-\xi }\xi^{-3/2}d\xi.
\]
All other parameters of $X$ are zero. Since $\int_0^\infty \xi \mu(d\xi)<\infty$, $X$ is a special martingale with jumps of finite variation. In particular, 
we have the L\'evy-It\^o  decomposition,
\[
X_t=X_0-\sqrt \pi \int_0^t X_{s-}ds+\left(\sum_{s\leq t}\Delta X_s-\int_0^t\xi \nu(d\xi, ds)\right).
\]
$X$ is homogeneous affine with characteristic exponent $\psi(t,u)$ which satisfies for $\Re(u)<1$ the initial value problem
\begin{equation}\label{eq ric ric}
\partial_t\psi(t,u)=1-\psi(t,u)-\sqrt{1-\psi(t,u)},\quad \psi(0,u)=u.
\end{equation}
It is straightforward to check that the unique solution of this equation is given by
\[
\psi(t,u)=1-\left((1-\sqrt{1-u})e^{-t/2}-1\right)^2.
\]
However, for $u= 1$, also  $g\equiv 1$ is a solution of \eqref{eq ric ric}.
\end{example}

\bibliographystyle{SIAM}

%\bibliography{references}

\end{document}